\documentclass[a4 paper]{article}  

\usepackage{graphicx}
\usepackage{epsfig} 
\usepackage{times} 
\usepackage{amsmath}
\usepackage{amssymb}  
\usepackage{color}
\usepackage{amsfonts}
\usepackage{subfigure}
\usepackage{multirow}
\usepackage{multicol}
\usepackage{siunitx}
\usepackage{overpic}
\usepackage{rotating}
\usepackage{booktabs}
\usepackage{threeparttable}

\usepackage{epstopdf}
\newtheorem{proof}{Proof}
\newtheorem{lemma}{Lemma}

\newtheorem{theorem}{Theorem}

\newtheorem{corollary}{Corollary}
\usepackage{xspace}

\newcommand{\cotwo}{\mbox{CO}\ensuremath{_2}\mbox{ }}
\newcommand{\cotwon}{\mbox{CO}\ensuremath{_2}}

\begin{document}

\title{\LARGE \bf Modeling and Estimation of the Humans' Effect on the \cotwo Dynamics Inside a Conference Room}

\author{Kevin Weekly, Nikolaos Bekiaris-Liberis and Alexandre M. Bayen
\thanks{*This research is funded by the Republic of Singapore's National Research Foundation through a grant to the Berkeley Education Alliance for Research in Singapore (BEARS) for the Singapore-Berkeley Building Efficiency and Sustainability in the Tropics (SinBerBEST) Program.}%
\thanks{K. Weekly, N. Bekiaris-Liberis, and A. M. Bayen are with the Department of Electrical Engineering and Computer Sciences, University of California, Berkeley, Berkeley, CA, 94720, USA. {\tt\small kweekly@eecs.berkeley.edu,  bekiaris-liberis@berkeley.edu, and bayen@berkeley.edu}}
}


\maketitle
\thispagestyle{empty}

\begin{abstract}
\baselineskip=1.7\normalbaselineskip
We develop a data-driven, {\em Partial Differential Equation-Ordinary Differential Equation} (PDE-ODE) model that describes the response of the {\em Carbon Dioxide} (\cotwon) dynamics inside a conference room, due to the presence of humans, or of a user-controlled exogenous source of \cotwon. We conduct two controlled experiments in order to develop and tune a model whose output matches the measured output concentration of \cotwo inside the room, when known inputs are applied to the model. In the first experiment, a controlled amount of \cotwo gas is released inside the room from a regulated supply, and in the second, a known number of humans produce a certain amount of \cotwo inside the room. For the estimation of the exogenous inputs, we design an observer, based on our model, using measurements of \cotwo concentrations at two locations inside the room. Parameter identifiers are also designed, based on our model, for the online estimation of the parameters of the model. We perform several simulation studies for the illustration of our designs. 
\end{abstract}
\baselineskip=1.8\normalbaselineskip

\section{Introduction}

\subsection{Motivation}

Reducing energy demand is an important component of smart building research. Building energy use is responsible for an increasing proportion of the total energy demand. In the United States, the proportion of building electricity consumption has raised to 40\% in 2005, from 33\% in 1980~\cite{doe2008energy} and in Singapore, buildings accounted for 31\% of the total electricity consumption for the year 2007 \cite{nea}. Thus, the problem of reducing building energy demand through advanced technologies and finer-tuned services has been the focus of ongoing research.  The knowledge of occupancy levels in discrete zones within a building offers the potential of significant energy savings when coupled with zonal control of building services \cite{agarwal,chao,erickson}, which is a motivation for the work presented in the present article.

A relatively unexplored approach for estimating the number of humans occupying discrete zones of office spaces, such as, for example, a conference room within a larger office space, is to model and estimate the effect of the \cotwo that is produced from humans on the total \cotwo concentration in the specific discrete zone (i.e., the conference room). The reason is that humans are the primary producers of \cotwo inside a building \cite{sepannen} and that \cotwo sensors are widely deployed in smart buildings (since \cotwo is an important quantity to observe in order to manage occupant comfort \cite{sepannen} and since this quantity can be measured using sensors which are cheap).

Modeling \cotwo dynamics is challenging, due to the complexity of air dynamics.  Most recently, two categories of models are used: Zonal models and {\em Computational Fluid Dynamics} (CFD) models.  CFD models provide the most rich and detailed view of air motion in a space, however, they are beset by arduous work in modeling the physical space (e.g. providing locations of all walls, furniture, and occupants) and identifying all parameters that are needed for the model.  CFD models also suffer from lengthy computation times to solve the necessary PDEs at a high resolution, especially near boundaries \cite{megri}, \cite{persily}. Zonal models relate the movement of air between discrete and well-mixed spaces, such as rooms and parts of rooms.   Generally, zonal models rely on ODE mass-balance laws between these spaces, which, in comparison to CFD models, can be solved very quickly \cite{megri}.  However, this comes at the expense of not modeling the distributed nature of airborne contaminant transfer within a single space, and complex local phenomena such as jets of air coming from a vent \cite{mora}.




Yet, for designing and implementing estimation algorithms for the \cotwo concentration, one has to develop a simple, and at the same time, accurate PDE-based model that retains the distributed character of the system. Based on this model, one can then design an observer for estimating the unknown \cotwo input that is produced from humans. The observer design has to be developed using the minimum number of sensors, in order to reduce cost and increase reliability. It is also crucial to develop online identifiers for the parameters of the model, since these parameters change with time due to their dependency on time-varying quantities such as heat generation \cite{baug}. 


\subsection{Literature}





Boundary observers for some classes of PDEs are constructed in \cite{flo1}, \cite{flo2}, \cite{krstic book siam}, \cite{mir1}, \cite{adrey observer} via backstepping. In \cite{moura1}, this methodology is applied for the estimation of the state-of-charge of batteries. Observer designs for time-delay systems with unknown inputs are presented in \cite{amin1}, \cite{litrico1}, \cite{litrico2}. Swapping identifiers, originally developed for parameter estimation of ODE systems \cite{ioannou}, \cite{krstic}, are constructed for parabolic PDEs in \cite{adaptive pde}, \cite{adaptive1}, \cite{adaptive2}, \cite{adaptive3}. In \cite{moura} this class of identifiers is employed for the identification of the state-of-health of batteries. Update laws for the estimation of unknown plant parameters and delays, in adaptive control of linear and nonlinear systems with input delays, are developed in \cite{me1}, \cite{me2}, \cite{delphine1}, \cite{delphine2}, \cite{delphine3}, \cite{delphine4}.


\subsection{Results}

We model the dynamics of the \cotwo concentration in the room using a convection PDE with a source term which is the output of a first-order ODE system driven by an unknown input which models the human's emission rate of \cotwon. The source term represents the effect of the humans on the \cotwo concentration in the room. In our experiments, we observe a delay in the response of the \cotwo concentration in the room to changes in the human's input. For this reason, the source term is a filtered version of the unknown input rather than the actual input. We assume that the unmeasured input from the humans has the form of a piecewise constant signal. This formulation is based on our experimental observation that humans contribute to the rate of change of the \cotwo concentration of the room with a filtered version of step-like changes in the rate of \cotwon.

The value of the PDE at the one boundary of its spatial domain indicates the \cotwo concentration inside the room at the location of the air supply. At this location, incoming air is entering the room, and hence, one can view the \cotwo concentration of the fresh incoming air as an input to the system. The value of the PDE at the other boundary of its spatial domain indicates the \cotwo concentration at the air return of the ventilation system. The air at this point is mixed with \cotwo that convects from the air supply towards the air return, and with \cotwo that is produced from humans. We consider the \cotwo concentration at this point as the output of our system. Any value of the PDE on an interior point of its spatial domain is an indicator of the concentration of \cotwo at the ceiling in a (non-ratiometric) normalized distance along an axis from the supply to the return vent.




We design an observer for the overall PDE-ODE system using boundary measurements (at the air supply and the air return). The observer estimates the unknown input from the humans, as well as the overall PDE state of our model. Our observer design and the proof of exponential stability of the observation error is based on the observer design from \cite{bekiaris} for linear systems with distributed sensor delays. We design PDE and ODE swapping identifiers for the three constant parameters of the overall PDE-ODE system, namely the convection speed of the PDE, the coefficient that multiplies the source term which affects the PDE state on its whole spatial domain, and the time constant of the ODE. We prove that all the identifiers are stable and that the identifier of the time constant of the ODE converge to its true value when the input from the humans is not zero. For the case in which the convection speed is known, we also prove that both the identifiers for the coefficient that multiplies the source term in the PDE and the time constant of the ODE converge to their true vales when the input from the humans is not zero.



\subsection{Structure of the Article}
In Section~\ref{secmodel}, we derive a coupled PDE-ODE model for the dynamics of the \cotwo concentration in the room. In Section~\ref{secobs}, we design an observer for the estimation of the total \cotwo that is generated by humans. We design parameter identifiers in Section~\ref{sec ide}, which are used for online estimation of the values of the model's parameters. 
 
\textit{Notation}: The spatial $L_2(0,1)$ norm is denoted by $\|\cdot\|$. The temporal norms are denoted by $\mathcal{L}_{\infty}$ and $\mathcal{L}_{2}$ for $t\ge0$.

\section{Model of the \cotwo Dynamics}
\label{secmodel}

Our model consists of a PDE and an ODE part. The ODE part is given by
\begin{eqnarray}
\dot{X}(t)&=&-aX(t)+V(t)\label{syss2}\\
\dot{V}(t)&=&0\label{sys3new},
\end{eqnarray}
where, $X(t)$, in ppm, models the source term of human \cotwo production on the relative concentration (in ppm) of the room in the local vicinity of the human (the evolution of which is described later on by a PDE), and $V(t)$ is a step-valued function, in ppm $\cdot$ s$^{-1}$, representing the level of the human \cotwo production rate within the vicinity of humans. Parameter, $\frac{1}{a}$, in units of $100$s, represents a time constant specifying how fast changes in occupancy affect the \cotwo concentration in the room, in the local vicinity of the human.


The ODE is coupled with a PDE that models the \cotwo concentration in the room given by
\begin{eqnarray}
u_t(x,t)&=&bu_x(x,t)+b_XX(t)\label{sys1new}\\
u(0,t)&=&U_{\rm e}-\Delta U(t),\label{syss}
\end{eqnarray}
with $\Delta U(t)=U(t)-U_{\rm e}$, where $u(x,t)$, in ppm, is the concentration of \cotwo in the room at a time $t\geq0$ s and for $0\leq x\leq 1$, $-b>0$, in $\frac{1}{100\textrm{s}}$, represents the rate of air movement in the room, and $b_X > 0$, in $\frac{1}{10^4\textrm{s}}$, specifies the rate of diffusion of \cotwo from the local vicinity of the human to the room. The spatial variable $x$ is unitless and represents a normalized distance along a horizontal axis that connects the air supply and air return. The air supply and air return are located at $x=0$ and $x=1$ respectively. Therefore, $u(0,t)$ is the \cotwo concentration inside the room at the location of the air supply and $u(1,t)$ is the \cotwo concentration inside the room at the location of the air return. The input $U(t)$ is the measured \si{ppm} concentration of the fresh incoming air. We do not simply specify the boundary condition at $x=0$ as $u(0,t)=U(t)$. The reason is that during our experiments we observe that a sudden drop in the measured \cotwo concentration at the air supply results in an increase of the \cotwo concentration at the air return. Our explanation for this effect is that a drop in \cotwo concentration at the supply from its equilibrium value corresponds to increased airflow at the vent, i.e. more fresh air gets mixed in the local vicinity. The increased airflow has the effect of pushing pockets of \cotwo air out of the return vent. One way to capture this effect is to multiply the difference of the \cotwo concentration from its equilibrium value $\Delta U(t)=U(t)-U_{\rm e}$ with minus one, where $U_{\rm e}$, in ppm, is the steady state input \cotwo concentration at the supply ventilation.

In Fig.~\ref{figillu},
\begin{figure}
\centering
\begin{overpic}[width=0.6\linewidth,tics=5]{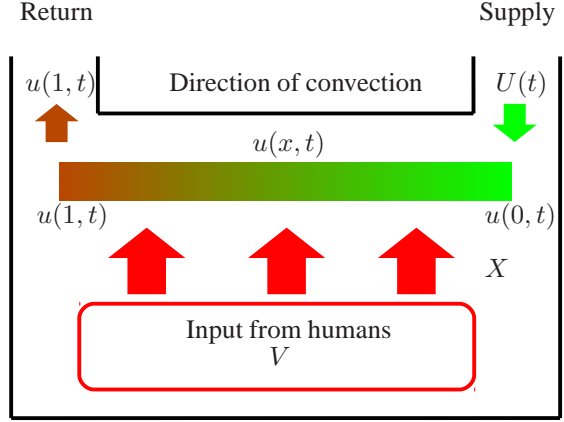}
\put(32,15){Input from humans}
\put(29,60){Direction of convection}
\put(29,57){}
\put(85,73){Supply}
\put(2,73){Return}
\put(47,10){$V$}
\put(44,49){$u(x,t)$}
\put(5,36){$u(1,t)$}
\put(86,36){$u(0,t)$}
\put(86,26){$X$}
\put(3,60){$u(1,t)$}
\put(88,60){$U(t)$}
\end{overpic}
\caption{The geometrical representation of our model. Fresh air ($U$) enters the room from the supply ventilation. Air near the ceiling ($u$) convects from the air supply to the air return vent. The humans produce \cotwo ($V$) which rises ($X$) to the ceiling.}
\label{figillu}
\end{figure}
we illustrate the geometrical representation of our model.  The PDE portion of the model, $u(x,t)$, represents convection of air from the air supply to the air return vent near the ceiling.  Note the absence of a diffusive term, which we have omitted since it plays a relatively minor role in dispersing indoor pollutants \cite{baug}.  We choose to model the \cotwo concentrations near the ceiling since this is where we see most effect from human-generated \cotwo.  This is explained by the fact that a warm breath from a human occupant acts as a ``bubble'' of gas that rises to the ceiling, since it is more buoyant than the ambient, cooler air.  Thus, the air coming from lower in the room is modeled as a source term on the PDE across its entire length. The ODE portion of the model is intended to model the fact that this bubble of air does not immediately rise to the ceiling but only gradually.




In Fig. \ref{fig1}
\begin{figure}[t]
\centering
\includegraphics[width=4in]{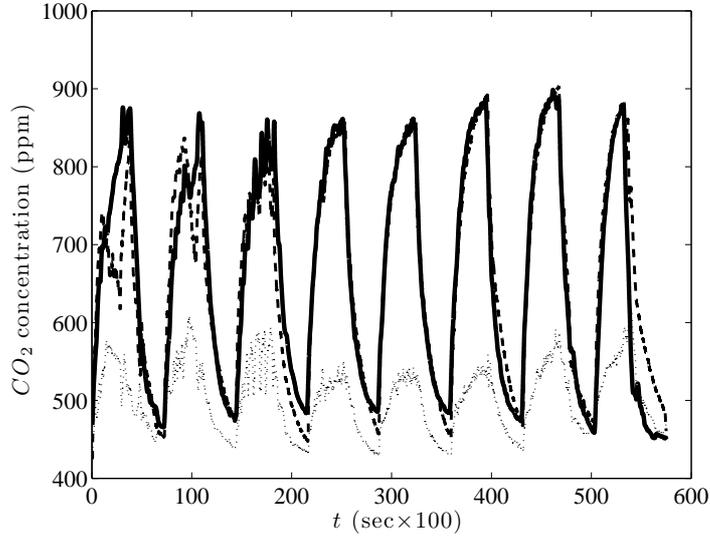}
\caption{Solid line: The simulated concentration of \cotwo at the air return $u(1,t)$ given by the model (\ref{syss2})--(\ref{syss}) for Experiment I. Dashed line: The concentration of the \cotwo at the air return measured by the \cotwo sensor. Dotted line: The concentration of \cotwo at the air supply measured by the \cotwo sensor.}
\label{fig1}
\end{figure}
 we show the concentration of \cotwo at the air return and the air supply measured by the \cotwo sensors for our first experiment in which we periodically release \cotwo every one hour. We also show the output $u(1,t)$ of our model with parameters as shown in Table \ref{table1}\footnote{In this section we manually tune the parameters of model (\ref{syss2})--(\ref{syss}) in order to match the measured \cotwo concentration at the air return with $u(1,t)$. In Section \ref{sec ide} we design identifiers for online identification of the parameters of model (\ref{syss2})--(\ref{syss}).} and initial condition $u(x,0)=400$ ppm. 
 \begin{table}[t]
\caption{Parameters of the Model (\ref{syss2})--(\ref{syss}) for Experiment I.}
\begin{center}
\resizebox{4in}{!} {
\begin{tabular}{ccc}
\hline\hline
Physical Paramater&Model parameter & Value\\
\hline
Convection coefficient $\left(\frac{1}{100\rm{s}}\right)$&$-b$&$0.8$ \\[2mm]
Source term coefficient $\left(\frac{1}{10^{4}\rm{s}}\right)$&$b_X$&$0.2$ \\[2mm]
Time constant of the human's effect $\left({100\rm{s}}\right)$& $\frac{1}{a}$&$10$\\[2mm]
Equilibrium concentration at the air return (ppm)& $U_{\rm e}$&$450$\\
\hline
\end{tabular}}
\label{table1}
\end{center}
\end{table}%
The input $V$ to our model, with which we emulate the behavior of the \cotwo that is released from the pump, is the square wave that is shown in Fig. \ref{fig2}.
 \begin{figure}[t]
\centering
\includegraphics[width=4in]{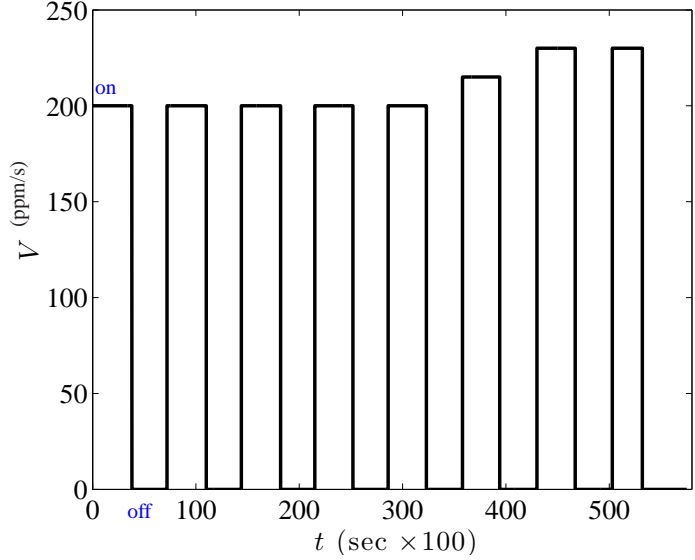}
\begin{picture}(0, 0) 
 \thicklines
\put(-285, 124){\begin{sideways}\footnotesize(ppm/s)\end{sideways}}
\put(-253, 174){\footnotesize \textcolor{blue}{on}}
\put(-241, 14){\footnotesize \textcolor{blue}{off}}
\end{picture}
\caption{The input $V$ to the model (\ref{syss2})--(\ref{syss}) from Experiment I modeling the concentration of \cotwo that is released from the pump. When $V=0$ the \cotwo pump is turned off and when $V\neq0$ the \cotwo pump is turned on.}
\label{fig2}
\end{figure}

In Fig. \ref{fig5} we show the \cotwo concentration from Experiment II measured from the \cotwo sensor and predicted from model (\ref{syss2})--(\ref{syss}) with parameters shown in Table \ref{table3}, initial condition $u(x,0)=400$ ppm, and input $V$ that is shown in Fig. \ref{fig6}, with which we emulate the behavior of the \cotwo that is produced by humans.
 \begin{table}[t]
\caption{Parameters of the Model (\ref{syss2})--(\ref{syss}) for Experiment II.}
\begin{center}
\resizebox{4in}{!} {
\begin{tabular}{ccc}
\hline\hline
Physical Paramater&Model parameter & Value\\
\hline
Convection coefficient $\left(\frac{1}{100\rm{s}}\right)$&$-b$&$0.8$ \\[2mm]
Source term coefficient $\left(\frac{1}{10^{4}\rm{s}}\right)$&$b_X$&$0.16$ \\[2mm]
Time constant of the human's effect $\left({100\rm{s}}\right)$& $\frac{1}{a}$&$10$\\[2mm]
Equilibrium concentration at the air return (ppm)& $U_{\rm e}$&$370$\\
\hline
\end{tabular}}
\label{table3}
\end{center}
\end{table}%
\begin{figure}[t]
\centering
\includegraphics[width=4in]{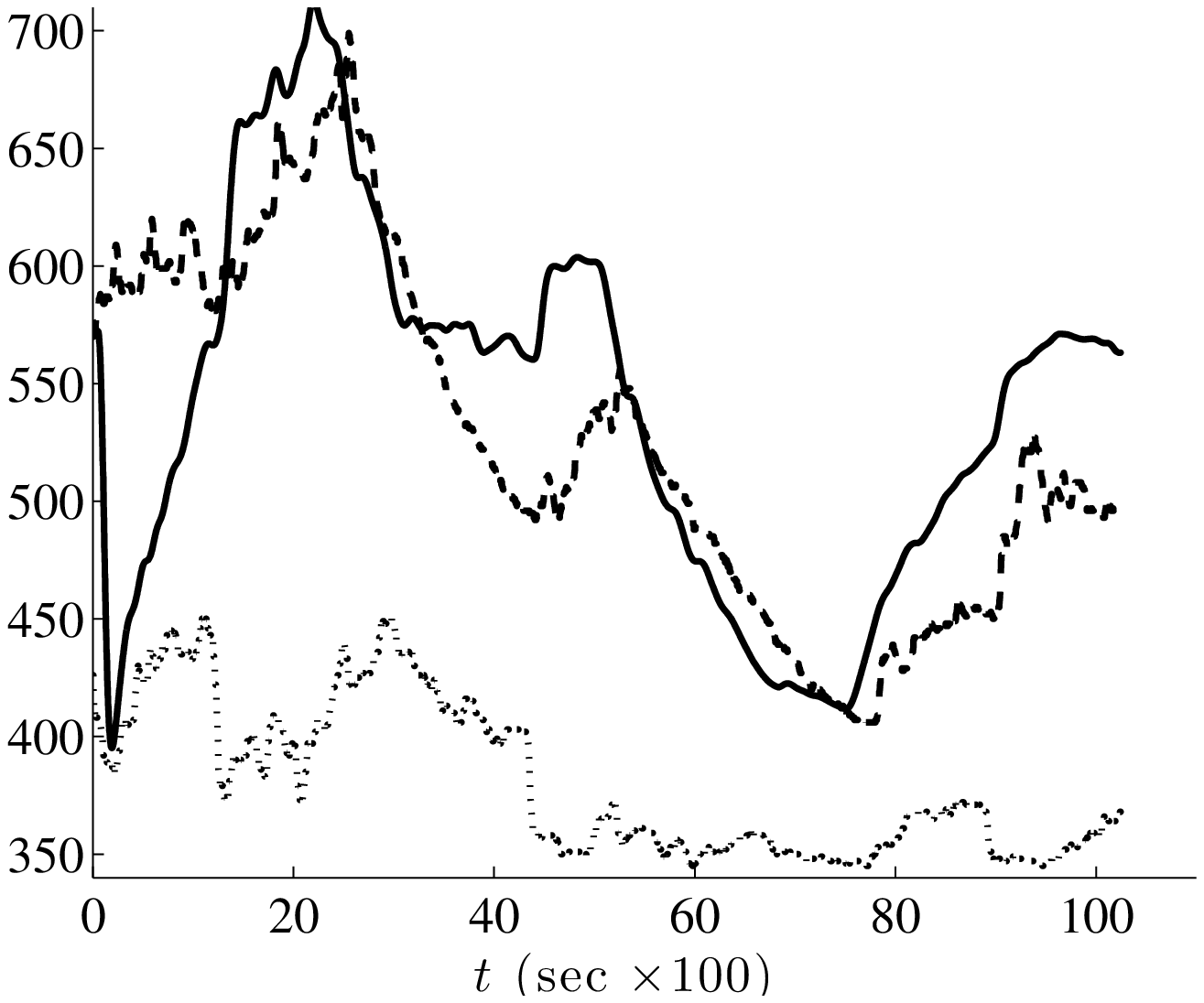}
\caption{Solid line: The simulated concentration of the \cotwo at the air return $u(1,t)$ given by the model (\ref{syss2})--(\ref{syss}) for Experiment II. Dashed line: The concentration of the \cotwo at the air return measured by the \cotwo sensor. Dotted line: The concentration of \cotwo at the air supply measured by the \cotwo sensor.}
\begin{picture}(0, 0) 
 \thicklines
\put(-140, 110){\begin{sideways}\footnotesize  \cotwo {concentration} (ppm)\end{sideways}}
\end{picture}
\label{fig5}
\end{figure}
\begin{figure}[t]
\centering
\includegraphics[width=4in]{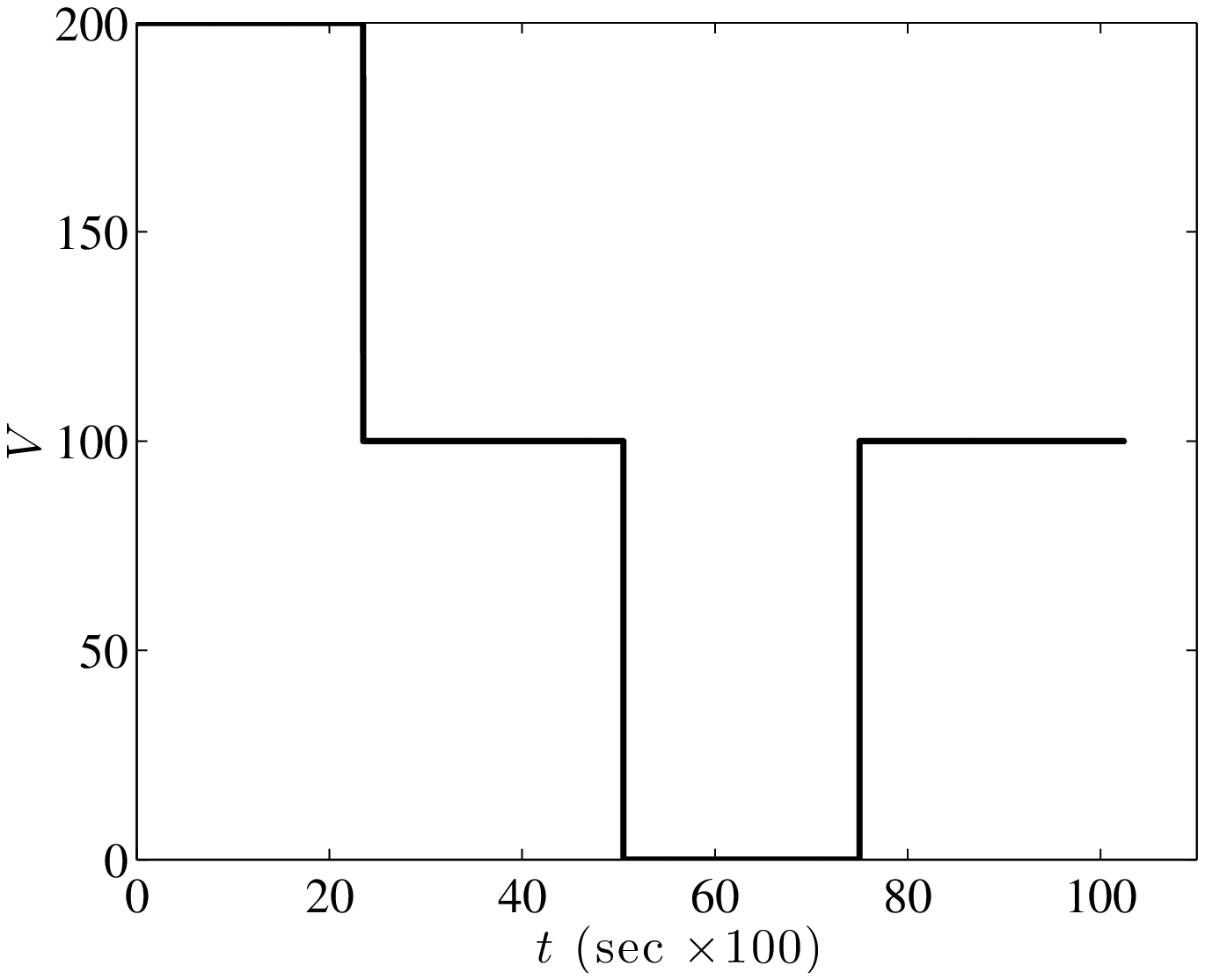}
\begin{picture}(0, 0) 
 \thicklines
\put(-285, 124){\begin{sideways}\footnotesize(ppm/s)\end{sideways}}
\put(-256, 215){\footnotesize \textcolor{blue}{Two humans}}
\put(-256, 205){\footnotesize \textcolor{blue}{in room}}
\put(-205, 128){\footnotesize \textcolor{blue}{One human}}
\put(-205, 118){\footnotesize \textcolor{blue}{in room}}
\put(-150, 38){\footnotesize \textcolor{blue}{No humans}}
\put(-150, 28){\footnotesize \textcolor{blue}{in room}}
\put(-103, 128){\footnotesize \textcolor{blue}{One human}}
\put(-103, 118){\footnotesize \textcolor{blue}{in room}}
\end{picture}
\caption{The input $V$ to the model (\ref{syss2})--(\ref{syss}) from Experiment II modeling the input concentration of \cotwo from the humans.}
\label{fig6}
\end{figure}



\section{Estimation of the Humans' Effect}
\label{secobs}
We construct an observer for the plant (\ref{syss2})--(\ref{syss}) assuming measurements of $u(1,t)$ and $U(t)$. We assume that the parameters of the model are known, since they can either be manually identified (as in Section \ref{secmodel}), or they can be identified using parameter identifiers (as in Section \ref{sec ide}). 
\subsection{Observer Design}
We consider the following observer which is a copy of the plant plus output injection
\begin{eqnarray}
\hat{u}_t(x,t)&=&b\hat{u}_x(x,t)+b_X\hat{X}(t)+p(x)\left(u(1,t)-\hat{u}(1,t)\right)\label{obs1}\\
\hat{u}(0,t)&=&-U(t)+2U_{\rm e}\\
\dot{\hat{X}}(t)&=&-a\hat{X}(t)+\hat{V}(t)+L_1\left(u(1,t)-\hat{u}(1,t)\right)\\
\dot{\hat{V}}(t)&=&L_2\left(u(1,t)-\hat{u}(1,t)\right).\label{obsn}
\end{eqnarray}

The following corollary is a consequence of Theorem 2 in \cite{bekiaris}.

\begin{corollary}
\label{thm1}
Consider the system (\ref{syss2})--(\ref{syss}) and the observer (\ref{obs1})--(\ref{obsn}) with
\begin{eqnarray}
p(x)&=&L_1\gamma_1(x)+L_2\gamma_2(x)\\
\gamma_1(x)&=&\frac{b_X}{a}\left(e^{-\frac{a}{b}x}-1\right)\label{defg1}\\
\gamma_2(x)&=&-\frac{b_X}{ba}x+\frac{b_X}{a^2}\left(1-e^{-\frac{a}{b}x}\right).\label{defg2}
\end{eqnarray}
Let $b_X\neq0$ and choose $L_1$, $L_2$ such that the matrix $A-\left[\begin{array}{ll}L_1\\L_2\end{array}\right]C$, where
\begin{eqnarray}
A&=&\left[\begin{array}{ll}-a&1\\0&0\end{array}\right]\label{defa}\\
C&=&\left[\begin{array}{ll}\gamma_1(1)&\gamma_2(1)\end{array}\right],\label{defc}
\end{eqnarray}
is Hurwitz. Then for any $u(x,0),\hat{u}(x,0)\in L_2(0,1)$, $X(0),\hat{X}(0),V(0),\hat{V}(0)\in\mathbb{R}$, there exist positive constants $\kappa$ and $\lambda$ such that the following holds for all $t\geq0$
\begin{eqnarray}
\Omega(t)&\leq&\kappa\Omega(0)e^{-\lambda t}\\
\Omega(t)&=&\int_0^1\left(u(x,t)-\hat{u}(x,t)\right)^2dx+\left(X(t)-\hat{X}(t)\right)^2+\left(V(t)-\hat{V}(t)\right)^2.
\end{eqnarray}
\end{corollary}

\begin{proof}
It is sufficient to show that if $b_X\neq0$ then the pair $\left(A,C\right)$ is observable (in which case one can choose $L_1$ and $L_2$ such that the matrix $A-\left[\begin{array}{ll}L_1\\L_2\end{array}\right]C$ is Hurwitz), which allow one to then use Theorem 2 in \cite{bekiaris}. The determinant of the observability matrix $O$ of the pair $(A,C)$ is $\textrm{det}(O)=\gamma_1(1)\left(\gamma_1(1)+a\gamma_2(1)\right)$. Using (\ref{defg1}), (\ref{defg2}) it follows that $\rm{det}\neq0$ whenever $b_X\neq0$.
\end{proof}

\section{Online Parameter Identification}
\label{sec ide}
We design swapping identifiers (see \cite{ioannou}, \cite{krstic} for the case of ODEs and \cite{moura}, \cite{adaptive pde} for the case of parabolic PDEs) for online identification of the parameters $b$, $b_X$ and $a$. We now assume that the ODE and PDE states are measured. Directly measuring these quantities in an actual implementation might be impractical. Yet, our online parameter identifiers can be in principle combined with a state-estimation algorithm in order to simultaneously perform state estimation and parameter identification, i.e., in order to design an adaptive observer (although, as it is discussed in Section \ref{concl}, for PDE systems this is highly nontrivial and there is no systematic approach for such a design). 
\subsection{Identifier Design}
 We deal first with the identification of $b$ and $b_X$. Define the ``estimation" error
\begin{eqnarray}
e(x,t)=u(x,t)-bv(x,t)-b_Xp(x,t)-\eta(x,t),\label{errordi}
\end{eqnarray}
between the measured state $u$ and the signals $v$, $p$, $\eta$, where $v$ is a filter for $u_x$, $p$ a filter for $X$ and $\eta$ is an input filter, given by
\begin{eqnarray}
v_t(x,t)&=&\hat{b}v_x(x,t)+u_x(x,t)\label{id1}\\
v(0,t)&=&0\\
p_t(x,t)&=&\hat{b}p_x(x,t)+X(t)\label{id2}\\
p(0,t)&=&0\label{id2s}\\
\eta_t(x,t)&=&\hat{b}\eta_x(x,t)-\hat{b}u_x(x,t)\\
\eta(0,t)&=&-U(t)+2U_e.\label{idnn}
\end{eqnarray}
The goal of the filters (\ref{id1})--(\ref{idnn}) is to convert the dynamic parametrization of the plant into a static one. This is the main attribute of the swapping identification method \cite{ioannou}, \cite{krstic}, \cite{adaptive pde}. Using the static relationship (\ref{errordi}) as a parametric model and defining the ``prediction error"
\begin{eqnarray}
\hat{e}(x,t)=u(x,t)-\hat{b}v(x,t)-\hat{b}_Xp(x,t)-\eta(x,t),\label{errordie}
\end{eqnarray}
the identifiers for $b$ and $b_X$ are given by the following gradient update laws with normalization
\begin{eqnarray}
\dot {\hat{b}}(t)&=&-\gamma_1\hat{b}(t)\textrm{Proj}_{\bar{b}}\left\{\frac{\int_0^1\hat{e}(x,t)v(x,t)dx}{1+\int_0^1v(x,t)^2dx+\int_0^1p(x,t)^2dx}\right\}\label{idnup}\\
\dot {\hat{b}}_X(t)&=&-\gamma_2\hat{b}(t)\frac{\int_0^1\hat{e}(x,t)p(x,t)dx}{1+\int_0^1v(x,t)^2dx+\int_0^1p(x,t)^2dx},\label{idn1up}
\end{eqnarray}
where the projector operator is defined as
\begin{eqnarray}
\mbox{Proj}_{\bar{b}}\{\tau\}&=& \left\{ \begin{array}{cc}0,& \mbox{if $\hat{b}=\bar{b}$ and $\tau>0$}   \\ \tau, & \mbox{otherwise} \end{array} \right.. \label{projector1}
\end{eqnarray}
and $\gamma_1,\gamma_2>0$, $\bar{b}<0$. The goal of the projection operator is to ensure that $\hat{b}<\bar{b}$. We design next an online identifier for $a$. Define the filters 
\begin{eqnarray}
\dot{\Omega}_0(t)&=&\bar{A}\left(\Omega_0(t)-X(t)\right)+V(t)\label{idnew1}\\
\dot{\Omega}(t)&=&\bar{A}\Omega(t)-X(t),\label{idnewn}
\end{eqnarray}
where $\bar{A}<0$. Defining the error
\begin{eqnarray}
\epsilon(t)=X(t)-\Omega_0(t)-\Omega(t){a},\label{error2}
\end{eqnarray}
the identifier for $a$ is
\begin{eqnarray}
\dot{\hat{a}}(t)&=&\gamma_3\frac{\hat{\epsilon}(t)\Omega(t)}{1+\Omega(t)^2}\label{idn2up}\\
\hat{\epsilon}(t)&=&X(t)-\Omega_0(t)-\Omega(t)\hat{a}(t),\label{error23}
\end{eqnarray}
where $\gamma_3>0$. Defining the estimation error of a parameter $m$ as
\begin{eqnarray}
\tilde{m}=m-\hat{m},
\end{eqnarray}
the following can be proved.

\begin{theorem}
\label{theorem2}
Consider system (\ref{syss2})--(\ref{syss}) with the update laws (\ref{id1})--(\ref{projector1}), (\ref{idnew1}), (\ref{idnewn}), (\ref{idn2up}), (\ref{error23}), and let $b<\bar{b}$, where $\bar{b}$ is known. Then for all $V(0)$, $X(0)$, $\hat{b}(0)$, $\hat{b}_X(0)$, $\hat{a}(0)$, $\Omega(0)$, $\Omega_0(0)$, $\in\mathbb{R}$, $v(x,0)$, $p(x,0)$, $\eta(x,0)$ $\in L_2(0,1)$, $u(x,0)\in H_1(0,1)$, and $U,\dot{U}\in\mathcal{L}_{\infty}$,
\begin{eqnarray}
&&\|u_x\|,\|u\|,V,X,\|v\|,\|p\|, \|\eta\|,\Omega,\Omega_0,\tilde{b}, \tilde{b}_X, \tilde{a}\in\mathcal{L}_{\infty}, \\
&&\dot{\hat{b}}, \dot{\hat{b}}_X, \dot{\hat{a}},\frac{\|\hat{e}\|}{\sqrt{1+\|v\|^2+\|p\|^2}}, \frac{|\hat{\epsilon}|}{\sqrt{1+\Omega^2}}\in\mathcal{L}_{2}\cap\mathcal{L}_{\infty}.
\end{eqnarray}
Moreover, if $\lim_{t\to\infty}V(t)\to c$, with $c\neq0$, then $\lim_{t\to\infty}\hat{a}(t)\to a$.
\end{theorem}

\begin{proof}
See Appendix A.
\end{proof}

In the special case in which $b$ is known\footnote{One could estimate the convection coefficient $b$ by measuring the delay in the response of the \cotwo concentration at the air return $u(1,t)$ to changes in the input \cotwo concentration at the air supply $U(t)$, when no exogenous sources of \cotwo are present. This can be either performed manually, or using signal processing techniques \cite{signal}, using the data that we collect from the two experiments.}, one does not need to design an update law for $b$ and the following can be proved. 

\begin{lemma}
\label{1}
Let $b$ be known. Consider the update law for $a$ given by (\ref{idn2up}) and the update law for $b_X$ given by
\begin{eqnarray}
\dot {\hat{b}}_X(t)&=&\gamma\frac{\int_0^1\hat{\zeta}(x,t)\mu(x,t)dx}{1+\int_0^1\mu(x,t)^2dx}\label{idn1upbx},
\end{eqnarray}
where
\begin{eqnarray}
\hat{\zeta}(x,t)&=&u(x,t)-\hat{b}_X\mu(x,t)-\xi(x,t)\label{errordibxest}\\
\mu_t(x,t)&=&b\mu_x(x,t)+X(t)\label{id2bx}\\
\mu(0,t)&=&0\label{id2sbx}\\
\xi_t(x,t)&=&b\xi_x(x,t)\\
\xi(0,t)&=&-U(t)+2U_e.\label{idnnbx}
\end{eqnarray}
Then for all $V(0)$, $X(0)$, $\hat{b}_X(0)$, $\hat{a}(0)$, $\Omega(0)$, $\Omega_0(0)$, $\in\mathbb{R}$, $u(x,0)$, $\mu(x,0)$, $\xi(x,0)$ $\in L_2(0,1)$, and $U\in\mathcal{L}_{\infty}$,
\begin{eqnarray}
&&\|u\|,V,X,\|\mu\|,\|\xi\|,\Omega,\Omega_0,\tilde{b}_X, \tilde{a}\in\mathcal{L}_{\infty},\quad \mbox{and}\quad \dot{\hat{b}}_X, \dot{\hat{a}}\in\mathcal{L}_{2}\cap\mathcal{L}_{\infty},\\
&&\frac{\|\hat{\zeta}\|}{\sqrt{1+\|\mu\|^2}}, \frac{|\hat{\epsilon}|}{\sqrt{1+\Omega^2}}\in\mathcal{L}_{2}\cap\mathcal{L}_{\infty}.
\end{eqnarray}
Moreover, if $V(0)=c$ with $c\neq0$, then $\lim_{t\to\infty}\hat{b}_X(t)\to b_X$ and $\lim_{t\to\infty}\hat{a}(t)\to a$.
\end{lemma}

\begin{proof}
See Appendix B.
\end{proof}

\section{Conclusions}
\label{concl}
In this article, we develop a PDE-ODE model that describes the dynamics of the \cotwo concentration in a conference room. We validate our model by conducting two different experiments. We design and validate an observer for the estimation of the unknown \cotwo input that is generated by humans. We also design online parameter identifiers for the online estimation of the parameters of the model.

Future work will address the problem of estimation of the actual human occupancy level using measurements of \cotwon. This is a highly nontrivial problem because humans' \cotwo generation rates can vary widely between different persons depending on current activity, diet, and body size \cite{persily}.

Another topic for future research is to combine the observer design with the update laws for the estimation of the parameters of the model. In other words, to design an adaptive observer \cite{ioannou}. Yet, in contrast to the finite-dimensional case, in the case of PDE systems this is far from trivial due to the lack of systematic procedures for the construction of state-transformations that can transform the original system to a system having an observer canonical form \cite{ioannou}, \cite{adaptive pde}. For this reason designing adaptive observers for PDE systems is possible only in special cases \cite{adaptive pde}. As an alternative one could resort to finite-dimensional approximations as it is done, for example, in \cite{moura adaptive observer}.


\section*{Acknowledgments}
The authors would like to thank William W. Nazaroff and Donghyun Rim for useful discussion and for pointing out some important literature, on the topics of indoor airflow modeling and of indoor contaminant source identification.







\setcounter{equation}{0}
\renewcommand{\theequation}{A.\arabic{equation}}
\section*{Appendix A}

\subsection*{Proof of Theorem \ref{theorem2}}
\label{proof thm1}
Using (\ref{sys1new}), (\ref{syss}) together with (\ref{id1}), (\ref{idnn}) one can show that the error (\ref{errordi}) satisfies 
\begin{eqnarray}
e_t(x,t)&=&\hat{b}e_x(x,t)\label{error eq}\\
e(0,t)&=&0.\label{error eq1}
\end{eqnarray}
Analogously, using (\ref{syss2}), along with (\ref{idnew1}), (\ref{idnewn}), it is shown that the error (\ref{error2}) satisfies
\begin{eqnarray}
\dot{\epsilon}(t)=\bar{A}\epsilon(t).\label{error eq2}
\end{eqnarray}
Consider the Lyapunov function 
\begin{eqnarray}
V(t)&=&\int_0^1(2-x)e(x,t)^2dx+\frac{1}{2}\epsilon(t)^2+\frac{1}{2\gamma_1}\tilde{b}(t)^2+\frac{1}{2\gamma_2}\tilde{b}_X(t)^2\nonumber\\
&&-\frac{\bar{A}}{2\gamma_3}\tilde{a}(t)^2.
\end{eqnarray}
Taking the derivative of $V$, using the fact that for a constant parameter $\dot{\tilde{m}}=-\dot{\hat{m}}$, the update laws (\ref{idnup}), (\ref{idn1up}), (\ref{idn2up}), the properties of the projector operator (see for example \cite{krstic}) and relations (\ref{error eq}), (\ref{error eq1}), (\ref{error eq2}) we get that
\begin{eqnarray}
\dot{V}(t)&\leq&\hat{b}(t)e(1,t)^2+\hat{b}(t)\int_0^1e(x,t)^2dx+\bar{A}\epsilon(t)^2+\hat{b}(t)\tilde{b}(t)\nonumber\\
&&\times\frac{\int_0^1\hat{e}(x,t)v(x,t)dx}{1+\int_0^1v(x,t)^2dx+\int_0^1p(x,t)^2dx}+\hat{b}(t)\tilde{b}_X(t)\nonumber\\
&&\times\frac{\int_0^1\hat{e}(x,t)p(x,t)dx}{1+\int_0^1v(x,t)^2dx+\int_0^1p(x,t)^2dx}+\bar{A}\tilde{a}(t)\frac{\hat{\epsilon}(t)\Omega(t)}{1+\Omega(t)^2}.
\end{eqnarray}
Using definitions (\ref{errordi}), (\ref{errordie}), and (\ref{error2}), (\ref{error23}) we get that 
\begin{eqnarray}
\hat{e}(x,t)&=&e(x,t)+\tilde{b}(t)v(x,t)+\tilde{b}_X(t)p(x,t)\label{errorsh}\\
\hat{\epsilon}(t)&=&\epsilon(t)+\tilde{a}(t)\Omega(t),\label{errorsh1}
\end{eqnarray}
and hence,
\begin{eqnarray}
\dot{V}(t)&\leq&\hat{b}(t)e(1,t)^2+\hat{b}(t)\int_0^1e(x,t)^2dx+\bar{A}\epsilon(t)^2\nonumber\\
&&+\hat{b}(t)\frac{\int_0^1\hat{e}(x,t)^2dx-\int_0^1\hat{e}(x,t)e(x,t)dx}{1+\int_0^1v(x,t)^2dx+\int_0^1p(x,t)^2dx}-\bar{A}\frac{-\hat{\epsilon}(t)^2+\hat{\epsilon}(t)\epsilon(t)}{1+\Omega(t)^2}.
\end{eqnarray}
Using Young's inequality and the fact that $\hat{b}(t)\leq\bar{b}<0$, for all $t\geq0$, we arrive at 
\begin{eqnarray}
\dot{V}(t)&\leq&\bar{b}e(1,t)^2+\frac{\bar{b}}{2}\int_0^1e(x,t)^2dx+\frac{\bar{A}}{2}\epsilon(t)^2+\nonumber\\
&&+\frac{\bar{b}}{2}\frac{\int_0^1\hat{e}(x,t)^2dx}{1+\int_0^1v(x,t)^2dx+\int_0^1p(x,t)^2dx}+\frac{\bar{A}}{2}\frac{\hat{\epsilon}(t)^2}{1+\Omega(t)^2}.\label{boundness}
\end{eqnarray}
Using relation (\ref{boundness}) we conclude that $\|e(t)\|$, $\epsilon$, $\tilde{b}$, $\tilde{b}_X$, $\tilde{a}$ are bounded, and that also $\frac{\sqrt{\int_0^1\hat{e}(x,t)^2dx}}{\sqrt{1+\int_0^1v(x,t)^2dx+\int_0^1p(x,t)^2dx}}$, $\frac{\left|\hat{\epsilon}(t)\right|}{\sqrt{1+\Omega(t)^2}}$ are square integrable. Therefore, using (\ref{idnup}), (\ref{idn1up}), (\ref{idn2up}), and the boundness of $\tilde{b}$ (which implies also the boundness of $\hat{b}$) one can conclude that $\dot{\hat{b}}$, $\dot{\hat{b}}_X$, $\dot{\hat{a}}$ are square integrable, and using (\ref{errorsh}), (\ref{errorsh1}) that $\frac{\|\hat{e}(t)\|}{\sqrt{1+\|v(t)\|^2+\|p(t)\|^2}}$ and $\frac{\left|\hat{\epsilon}(t)\right|}{\sqrt{1+\Omega(t)^2}}$ are bounded. Therefore, using the update laws (\ref{idnup}), (\ref{idn1up}), (\ref{idn2up}) one can conclude that $\dot{\hat{b}}$, $\dot{\hat{b}}_X$, $\dot{\hat{a}}$ are also bounded. Using relations (\ref{syss2}), (\ref{sys3new}), (\ref{idnew1}), (\ref{idnewn}) one can conclude that $V,X,\Omega,\Omega_0\in\mathcal{L}_{\infty}$. Using (\ref{sys1new}), (\ref{syss}) one can conclude that 
\begin{eqnarray}
u_{xt}(x,t)&=&bu_{xx}(x,t)\label{fuc1}\\
u_x(0,t)&=&-\frac{\dot{U}(t)+b_XX(t)}{b}.\label{fuc2}
\end{eqnarray}
With a Lyapunov functional as
\begin{eqnarray}
V^*(t)&=&\int_0^1(2-x)u(x,t)^2dx+\int_0^1(2-x)u_x(x,t)^2dx\nonumber\\
&&+\frac{b\bar{b}}{16}\int_0^1(2-x)v(x,t)^2dx+\int_0^1(2-x)p(x,t)^2dx,
\end{eqnarray}
we get along (\ref{sys1new}), (\ref{syss}), (\ref{fuc1}), (\ref{fuc2}), (\ref{id1}), (\ref{id2s}), after using integration by parts and Young's inequality that 
\begin{eqnarray}
\dot{V}^*(t)&\leq&\frac{b}{2}\int_0^1u(x,t)^2dx+ \frac{b}{2}\int_0^1u_x(x,t)^2dx+\frac{\bar{b}}{2}\int_0^1v(x,t)^2dx\nonumber\\
&&+\frac{\bar{b}}{2}\int_0^1p(x,t)^2dx -\frac{8}{\bar{b}}X(t)^2\left(b_X+1\right) -b\left(-U(t)+2U_{\rm e}\right)^2\nonumber\\
&&-\frac{\left(\dot{U}(t)+b_XX(t)\right)^2}{\bar{b}},
\end{eqnarray}
and hence,
\begin{eqnarray}
\dot{V}^*(t)&\leq&\frac{\bar{b}}{4}V^*(t)+M\left(X(t)^2+U(t)^2+\dot{U}(t)^2+U_{\rm e}^2\right),
\end{eqnarray}
where $M=-\frac{8}{\bar{b}}\left(b_X+1\right)-8b-2\frac{b_X^2+1}{\bar{b}}$. It follows since $U,X,\dot{U}\in\mathcal{L}_{\infty}$ that $\|u\|$, $\|u_x\|$, $\|v\|$, $\|p\|$, $\in\mathcal{L}_{\infty}$. Using (\ref{errordi}) it follows that $\|\eta\|\in\mathcal{L}_{\infty}$.

We show next that $\lim_{t\to\infty}\tilde{a}(t)=0$ when $\lim_{t\to\infty}V(t)=c\neq0$. Using (\ref{syss2}), (\ref{idnewn}), and the fact that $a>0$, $\bar{A}<0$, one can conclude that $\Omega$ is bounded, when $V$ is bounded, with $\lim_{t\to\infty}\Omega(t)=\frac{c}{a\bar{A}}$, and hence, it is sufficient to show that $\lim_{t\to\infty}\hat{\epsilon}(t)=0$, since then, one can conclude using (\ref{error eq2}) and (\ref{errorsh1}) that $\lim_{t\to\infty}\tilde{a}(t)=0$. Using an alternative to Barbalat's Lemma from \cite{liu} it is sufficient to show that $\frac{dG(t)}{dt}$, where $G(t)=\frac{\hat{\epsilon}(t)^2}{{1+\Omega(t)^2}}$, is bounded. We have that $G$ satisfies the relation $\frac{dG(t)}{dt}=2\frac{\hat{\epsilon}(t)}{\sqrt{1+\Omega(t)^2}}\frac{\left(\bar{A}\epsilon(t)-\dot{\hat{a}}(t)\Omega(t)+\tilde{a}(t)\left(\bar{A}\Omega(t)-X(t)\right)\right)}{\sqrt{1+\Omega(t)^2}}-\frac{\hat{\epsilon}(t)^2}{1+\Omega(t)^2}\frac{\Omega(t)\left(\bar{A}\Omega(t)-X(t)\right)}{1+\Omega(t)^2}$. Since $\frac{\hat{\epsilon}(t)}{\sqrt{1+\Omega(t)^2}}$, $\epsilon$, $\dot{\hat{a}}$, $\tilde{a}$ are bounded, using (\ref{syss2}) one can conclude that $X$ is also bounded (when $V$ is bounded), and hence, $\frac{dG(t)}{dt}$ is bounded. 

\setcounter{equation}{0}
\renewcommand{\theequation}{B.\arabic{equation}}
\section*{Appendix B}
\subsection*{Proof of Lemma \ref{1}}
Using (\ref{sys1new}), (\ref{syss}) together with (\ref{id2bx}), (\ref{idnnbx}) it is shown that for the error
\begin{eqnarray}
\zeta(x,t)&=&u(x,t)-b_X\mu(x,t)-\xi(x,t)\label{errordibx},
\end{eqnarray}
it holds that
\begin{eqnarray}
\zeta_t(x,t)&=&b\zeta_x(x,t)\label{error eqbx}\\
\zeta(0,t)&=&0.\label{error eq1bx}
\end{eqnarray}
Similarly to the proof of Theorem \ref{theorem2}, using the fact that $\zeta(x,t)-\hat{\zeta}(x,t)=-\tilde{b}_X\mu(x,t)$, for the Lyapunov function 
\begin{eqnarray}
V(t)&=&-\frac{1}{b}\int_0^1(2-x)\zeta(x,t)^2dx+\frac{1}{2}\epsilon(t)^2+\frac{1}{2\gamma}\tilde{b}_X(t)^2-\frac{\bar{A}}{2\gamma_3}\tilde{a}(t)^2,
\end{eqnarray}
along the solutions of (\ref{error eqbx}), (\ref{error eq1bx}), (\ref{error eq2}), (\ref{idn1upbx}), (\ref{idn2up}) it holds that
\begin{eqnarray}
\dot{V}(t)&\leq&-\zeta(1,t)^2-\frac{1}{2}\int_0^1\zeta(x,t)^2dx+\frac{\bar{A}}{2}\epsilon(t)^2-\frac{1}{2}\frac{\int_0^1\hat{\zeta}(x,t)^2dx}{1+\int_0^1\mu(x,t)^2dx}\nonumber\\
&&+\frac{\bar{A}}{2}\frac{\hat{\epsilon}(t)^2}{1+\Omega(t)^2}.\label{boundnessbx}
\end{eqnarray}
We only prove that $\lim_{t\to\infty}b_X(t)\to b_X$. The rest of the lemma is proved using the same arguments with the proof of Theorem \ref{theorem2}. Using (\ref{id2bx}), (\ref{id2sbx}) it is shown that the variable $\bar{\mu}(x,t)=\mu(x,t)+\frac{c}{ab}x$ satisfies $\bar{\mu}_t(x,t)=b\bar{\mu}_x(x,t)+X(t)-\frac{c}{a}$, $\bar{\mu}(0,t)=0$. Using the fact that $V(0)=c$ we get from (\ref{syss2}), (\ref{sys3new}) that $\tilde{X}(t)=e^{-a t}\tilde{X}(0)$, where $\tilde{X}=X-\frac{c}{a}$. Therefore,  $\tilde{X}\in\mathcal{L}_2$. For the function $V_2(t)=V_1(t)+\frac{9}{-2ab}\tilde{X}(t)^2$, where $V_1(t)=\int_0^1(2-x)\bar{\mu}(x,t)^2dx$ it holds that
\begin{eqnarray}
\dot{V}_2(t)\leq \frac{b}{4}V_1(t)+ \frac{1}{b}\tilde{X}(t)^2.
\end{eqnarray}
Therefore, $\sqrt{V_1}\in\mathcal{L}_2\cap\mathcal{L}_{\infty}$, and hence $\sqrt{\int_0^1\bar{\mu}(x,t)^2dx}\in\mathcal{L}_2$. Using the facts that $\int_0^1\hat{\zeta}(x,t)^2dx= \frac{\int_0^1\hat{\zeta}(x,t)^2dx}{1+\int_0^1\mu(x,t)^2dx}\left(1+\int_0^1\mu(x,t)^2dx\right)$, that $\sqrt{V_1}\in\mathcal{L}_{\infty}$, and that $\frac{\sqrt{\int_0^1\hat{\zeta}(x,t)^2dx}}{\sqrt{1+\int_0^1\mu(x,t)^2dx}}\in\mathcal{L}_2$, it also follows that $\sqrt{\int_0^1\hat{\zeta}(x,t)^2dx}\in\mathcal{L}_2$. Writing $\zeta(x,t)-\hat{\zeta}(x,t)+\tilde{b}_X\bar{\mu}(x,t)=\frac{c}{ab}x\tilde{b}_X$ we get that $\frac{c}{2ab}\left|\tilde{b}_X\right|\leq\sqrt{\int_0^1\zeta(x,t)^2dx}+\sqrt{\int_0^1\hat{\zeta}(x,t)^2dx}+\left|\tilde{b}_X\right|\sqrt{\int_0^1\bar{\mu}(x,t)^2dx}$, and hence, since also $\sqrt{\int_0^1\zeta(x,t)^2dx}\in\mathcal{L}_2$ and $\tilde{b}_X\in\mathcal{L}_{\infty}$, and $c\neq0$ we get that $\tilde{b}_X\in\mathcal{L}_{2}$. Since $\tilde{b}_X,\dot{\tilde{b}}_X\in\mathcal{L}_{\infty}$, one can conclude that $\frac{d\tilde{b}_X^2(t)}{dt}\in\mathcal{L}_{\infty}$, and hence, from the alternative to Barbalat's Lemma from \cite{liu} we conclude that $\lim_{t\to\infty}\tilde{b}_X(t)\to0$.









\end{document}